\documentclass[12pt]{article}
\usepackage[english]{babel}
\usepackage{epsfig}
\usepackage{amssymb,amsmath,amsfonts,soul,amsthm,enumerate}
\usepackage{graphicx}
\usepackage{url}
\usepackage{mathrsfs}
\usepackage{xcolor}
\usepackage[section]{placeins}
\newtheorem{theorem}{Theorem}

\newtheorem{definition}{Definition}

\newtheorem{lemma}{Lemma}

\newtheorem{corollary}{Corollary}

\newtheorem{proposition}{Proposition}

\newtheorem{remark}{Remark}

\providecommand{\R}{\mathbb{R}}
\newcommand{\Lo}{\mathbb{L}}
\newcommand{\binomial}[2]{\left(\!\!\begin{array}{c} #1 \\  #2 \end{array}\!\!\right)}

\def\@evenhead{\raisebox{0pt}[\headheight][0pt]{%
\vbox{{\hbox to \textwidth{\thepage\leftmark\strut}}\hrule}}}
\def\@evenfoot{}
\def\@oddfoot{}

\author{Daniel Gerth\footnotemark[1],\ \ 
 Bernd Hofmann\footnotemark[1], \  
Christopher Hofmann\footnotemark[1],\\ 
and  Stefan Kindermann\footnotemark[2]}
\title{The Hausdorff Moment Problem in the light of ill-posedness of type I}
\date{}

\begin{document}
\renewcommand{\thefootnote}{\fnsymbol{footnote}}
\maketitle
\footnotetext[1]{Chemnitz University of Technology, \\
Faculty of Mathematics, 09107 Chemnitz, Germany,\\
Email: {\tt daniel.gerth@mathematik.tu-chemnitz.de,\\ 
bernd.hofmann@mathematik.tu-chemnitz.de,\\ 
christopher.hofmann@mathematik.tu-chemnitz.de}}
\footnotetext[1]{Johannes Kepler University Linz, Industrial Mathematics Institute,\\
Altenbergerstra{\ss}e 69, A-4040 Linz, Austria,\\
Email: {\tt kindermann@indmath.uni-linz.ac.at}}



\begin{abstract} 
{The Hausdorf moment problem (HMP) over the unit interval in an $L^2$-setting is a classical example of an ill-posed inverse problem. Since various applications can be rewritten  in terms of the HMP, it has gathered significant attention in the literature. From the point of view of regularization it is of special interest because of the occurrence of a non-compact forward operator with non-closed range. Consequently, HMP constitutes one of few examples of a linear ill-posed problem of type~I in the sense of Nashed. In this paper we highlight this property and its consequences, for example, the existence of
a infinite-dimensional subspace of stability.
On the other hand, we show conditional stability estimates for the HMP
in Sobolev spaces that indicate severe ill-posedness for the full recovery of a function from its moments, because H\"{o}lder-type stability can be excluded. However, the associated recovery
of the function value at the rightmost point of the unit interval is stable of H\"{o}lder-type in an $H^1$-setting. We moreover discuss stability estimates for the truncated HMP, where the forward operator becomes compact. Some numerical case studies illustrate the theoretical results and complete the paper.}
\end{abstract} 
\vskip 0.2cm


\vskip 0.2cm


\vskip 0.3cm

\setcounter{figure}{0}

\section{Introduction}

In the past few decades, there has been developed a comprehensive theory and practice for the stable approximate solution of ill-posed linear operator equations
\begin{equation} \label{eq:opeq}
A\, x\,=\,y\,,
\end{equation}
modelling inverse problems with bounded linear forward operators $A:X \to Y$ mapping between infinite-dimensional separable Hilbert spaces $X$ and $Y$. We refer in this context to the corresponding chapters of the monographs \cite{BakuKok04,EHN96,Groetsch07,Hanke17,Hof99,Ito15,Kaba12,Kirsch11,LuPer13,Nair09,Tik95}. Hadamard (cf.~\cite{Hadamard23}) introduced the concept of well-posedness applicable to an operator equation \eqref{eq:opeq} that requires injectivity, surjectivity and continuous invertibility of the operator $A$. If $A$ violates at least one of the three conditions, then \eqref{eq:opeq} is ill-posed in the sense of Hadamard.
For linear problems \eqref{eq:opeq}, which are ill-posed in the sense of Hadamard, Zuhair Nashed focused in \cite{Nashed87}
on the stability aspect, which is the most important one for the numerical analysis. Moreover, ibid he introduced the distinction of two types of ill-posedness. The following definition (cf.~\cite[Def.~2]{HofPla18}) is a consequence of this introduction.
\begin{definition}[well-posedness vs.~ill-posedness] \label{def:Nashed}
We call a linear operator equation \eqref{eq:opeq}  well-posed in the sense of Nashed if  the range $\mathcal{R}(A)$ of the bounded linear operator $A$ is a closed subset of $Y$, consequently ill-posed in the sense of Nashed if the range is not closed, i.e., $\mathcal{R}(A) \not= \overline{\mathcal{R}(A)}^{Y}$.
In the ill-posed case, the equation \eqref{eq:opeq} is called ill-posed of type~I if the range $\mathcal{R}(A)$ contains an infinite-dimensional closed subspace, and ill-posed of type~II otherwise.
\end{definition}
\begin{remark} \label{rem:rem2}
{\rm
A necessary condition for ill-posedness in the sense of Definition~\ref{def:Nashed} is that
\begin{equation}\label{eq:infdim}
{\rm dim}\,(\mathcal{R}(A))=\infty\,.
\end{equation}
On the other hand, a criterion differentiating well-posedness from ill-posedness in the Hilbert space setting is delivered by the Moore-Penrose pseudoinverse
$$A^\dagger: \mathcal{R}(A) \oplus \mathcal{R}(A)^\perp \subset Y \to \mathcal{N}(A)^\perp \subset X$$ of the forward operator $A$. If and only if the linear operator $A^\dagger$ is a bounded one, the range $\mathcal{R}(A)$ is closed and hence the operator equation \eqref{eq:opeq} is well-posed (see also \cite[\S2.1]{EHN96}).
For Hilbert spaces $X$ and $Y$, the equation \eqref{eq:opeq} is, under the condition \eqref{eq:infdim}, ill-posed in the sense of Nashed of type~II  if and only if the operator $A$ is compact  \cite[Thm.~4.6]{Nashed87}. Well-posedness in the sense
of Definition~\ref{def:Nashed} does, however,  not exclude the case of non-injective $A$ possessing non-trivial null-spaces $\mathcal{N}(A)$.
Note that an analog to Definition~\ref{def:Nashed} in Banach spaces, but only for injective $A$, has been discussed in \cite{Flemming15}. Extended discussions of the non-injective case can be found in \cite[\S1.2.4]{Flemming18}.
}
\end{remark}

The vast majority of linear ill-posed operator equations \eqref{eq:opeq} with applications in natural sciences and imaging have a compact forward operator $A$ and are hence of type~II in the sense of Nashed. Mostly, these problems can be written as linear Fredholm integral equations of the first kind
\begin{equation} \label{eq:inteq}
[A\,x](s):=\int \limits _\Omega k(s,t)\,x(t)\,dt = y(s)\qquad (s \in \Sigma),
\end{equation}
with a non-degenerate kernel $k: \Sigma \times \Omega \subset \mathbb{R}^{d_1} \times  \mathbb{R}^{d_2} \to \mathbb{R}$ that implies a compact forward operator $A: X=L^2(\Omega) \to Y=L^2(\Sigma)$
for sufficiently regular and bounded subsets $\Sigma$ of $\mathbb{R}^{d_1}$ and $\Omega$ of $\mathbb{R}^{d_2}$. If, for example, the kernel $k$ is square integrable, i.e., $k \in L^2(\Sigma \times \Omega)$, then $A$ is compact and even of Hilbert-Schmidt type. Such linear operator equations that are ill-posed of type~II with compact operator $A$ are characterized by the singular system $\{\sigma_i,u_i,v_i\}_{i=1}^\infty$ (cf.~\cite[\S2.2]{EHN96}) with singular values
$$\|A\|=\sigma_1 \ge \sigma_2 \ge ... \ge \sigma_i \ge \sigma_{i+1} \ge ... \to 0 \quad \mbox{as} \quad i \to \infty.$$
Then the corresponding degree of ill-posedness of equation \eqref{eq:opeq} is represented by the decay rate of the singular values, and we refer in this context for details to \cite[\S3.1.6]{Hof99} or \cite[Def.~8]{HofKin10}.

Discussions about the smaller class of operator equations \eqref{eq:opeq} being ill-posed of type~I with non-compact operators $A$ are often focusing on multiplication operators
\begin{equation} \label{eq:mult}
[A\,x](t):= m(t)\,x(t) = y(s)\qquad (t \in \Omega),
\end{equation}
with $A:X=L^2(\Omega) \to Y=L^2(\Omega)$, where the multiplier function $m \in L^\infty(\Omega)$ possesses essential zeros (cf., e.g.,~\cite{Hof06,HofFlei99,MNH20}). In the present paper, however, we will consider another ill-posed problem of type~I, which occurs when one tries to recover a real function over $[0,1]$ from the infinite sequence of its moments. Then the linear operator $A$ maps from the separable Hilbert space $X=L^2(0,1)$ of quadratically integrable real functions over the unit interval $[0,1]$ to the separable Hilbert sequence space $Y=\ell^2$. This is a variant of the Hausdorff moment problem (HMP). It will be introduced in {Section~\ref{sec:HMPintro}} and discussed further with respect to analysis
and numerics in the subsequent sections. {Section~\ref{sec:proof}} delivers a proof of the ill-posedness of type~I and some further characterization of the HMP. The truncated version of the HMP using a finite number of moments is introduced in {Section~\ref{sec:truncated}} and further discussed with respect to error estimates in {Section~\ref{sec:errest}}. Relations to other problems and inversion are presented in
{Section~\ref{sec_relations}}. On the other hand, {Section~\ref{sec:stabest}} completes the analysis of the HMP with a series of stability results under Sobolev-type smoothness assumptions, which indicate the severe ill-posedness of the full recovery of a function from its moment. However, the associated recovery of the function value at the rightmost point of the unit interval is stable of H\"older-type.
{Section~\ref{sec:casestudies}} complements and illustrates the theory by means of some numerical case studies. To prepare for all this, we complete in the following the introductory section with some additional assertions on ill-posed problems of type~I.

\medskip

There is a simple relation between the ill-posedness of the linear operator equation \eqref{eq:opeq} with forward operator $A:X \to Y$
mapping between Hilbert spaces and the associated equation
\begin{equation} \label{eq:adj}
A^*\, y\,=\,x\,,
\end{equation}
with the adjoint operator $A^*:Y \to X$ to $A$ as forward operator.
\begin{proposition}
The operator equation \eqref{eq:opeq} is ill-posed of type~I in the sense of Nashed if and only if the adjoint equation
\eqref{eq:adj} is ill-posed of type~I.
\end{proposition}
\begin{proof}
The closed range theorem states that $\mathcal{R}(A)$ is closed if and only if
$\mathcal{R}(A^*)$ is closed. Thus, we have the result that the operator equation \eqref{eq:opeq} is ill-posed
in the sense of Nashed if and only if the adjoint equation \eqref{eq:adj} is ill-posed.
Moreover, Schauder's theorem says that $A$ is compact if and only if
$A^*$ is. Finally, as stated above,
it is  a classical result that, for mappings between
Hilbert spaces, an ill-posed problem is of type I if and only if
$A$ is non-compact. All together this proves the proposition.
\end{proof}

The following proposition presents a further simple equivalent formulation of ill-posedness of type~I.
\begin{proposition}
Let the problem~\eqref{eq:opeq} be ill-posed in the sense of Nashed.
Then it is ill-posed of type I if and only if
there exist a constant $C_1>0$ and an infinite-dimensional subspace $X_1$ of $X$
such that
\begin{equation}\label{eq:lower}
 \|x\|_X \leq C_1 \|A x\|_Y \qquad \forall x \in X_1.  \end{equation}
\end{proposition}
\begin{proof}
Let $A$ be ill-posed of type I. Then there exists a closed infinite dimensional subspace $Z$ of $Y$
with $Z \subset \mathcal{R}(A)$.
Consider the operator $\tilde{A} = A|_{N(A)^\bot}$. Then $\tilde{A}$
is injective and clearly, due to $\mathcal{R}(A) = \mathcal{R}(\tilde{A})$, its range contains $Z$. Thus,
by the open mapping theorem, $\tilde{A}$ has a bounded inverse on $Z$, i.e.,
 for all $y\in Z$, $\|\tilde{A}^{-1} y\|_X \leq C_1 \|y\|_Y$. Since $Z \subset \mathcal{R}(A)$
 we can replace any $y = \tilde{A} x = A x$ with the corresponding $x \in \mathcal{N}(A)^\bot$ yielding
\[ \|x\|_X \leq C_1 \|A x\|_Y \qquad
 \forall x \in \tilde{A}^{-1} Z =: X_1. \]
It is not difficult to show that $X_1$ is closed and infinite-dimensional.

Conversely, if \eqref{eq:lower} holds, then $A$ cannot be compact as we may choose an
orthonormal basis $\{e_i\}_{i = 1}^\infty$ in $X_1$ with $\|e_i\|_X= 1 \;\forall\,i \in \mathbb{N}$. The corresponding
image sequence satisfies for any $i\not = j$ that
 $$\|A e_i - Ae_j\|_Y^2 \geq C_1^{-2} \|e_i -e_j\|_X^2 =2 C_1^{-2}.$$
 Thus, the sequence $\{e_i\}_{i = 1}^\infty$  cannot
 have a convergent subsequence. Hence $A$ must be non-compact.
\end{proof}

We thus conclude the following result:
\begin{corollary} \label{cor:equiv}
For a bounded linear operator $A:X \to Y$
between Hilbert spaces $X$ and $Y$ with non-closed range $\mathcal{R}(A)$, the following assertions are  equivalent:
\begin{itemize}
\item[(a)] The operator equation \eqref{eq:opeq} is ill-posed of type I.
\item[(b)] There exist a constant $C_1$ and an infinite closed subspace $X_1$ with
\begin{equation}\label{eq:t1alhpa} \|x\|_X \leq C_1 \|A x\|_Y \qquad
 \forall x \in X_1.  \end{equation}
\item[(c)] There exist a constant $D_1$ and an infinite closed subspace $Y_1$ with
\begin{equation}\label{eq:t1beta}  \|y\|_Y \leq D_1 \|A^*y\|_Y \qquad
 \forall y \in Y_1. \end{equation}
\end{itemize}
\end{corollary}

\section{The Hausdorff moment problem over the unit interval} \label{sec:HMPintro}

Now we are going to introduce the Hausdorff moment problem (HMP) over the unit interval, which goes back to Hausdorff's article \cite{Hausdorff23},  as a special case of the linear operator equation \eqref{eq:opeq}. Discussions on alternative moment problems and variants of the Hausdorff moment problem, including the multi-dimensional case, are for example given in \cite{Angetal99,Lin17,Kazemi17,Tagliani11,Tamarkin43}.

\begin{definition}[Hausdorff moment problem] \label{def:HMP}
We call the inverse problem of solving the linear operator equation \eqref{eq:opeq} the Hausdorff moment problem (HMP) if the forward operator $A$ maps from the separable Hilbert space $X=L^2(0,1)$
into the separable Hilbert space $Y=\ell^2$  and attains the form
\begin{equation} \label{eq:A}
[A\,x]_j:=\int \limits _0^1 t^{j-1}\,x(t)\,dt \qquad (j=1,2,...).
\end{equation}
Precisely, a real function $x$ with support on $[0,1]$ has to be recovered from the infinite (countable) sequence of its moments.
\end{definition}

The assertions of the following proposition on the HMP-forward operator $A$ are either taken from \cite{Inglese92} with references therein or immediately evident.

\begin{proposition} \label{pro:Inglese}
For the operator $A: L^2(0,1) \to \ell^2$, which was introduced by formula~\eqref{eq:A},  we have the following properties:
$A$ is an injective and bounded linear operator with $\|A\|_{\mathcal{L}(L^2(0,1),\ell^2)}=\sup\limits_{0 \not= x \in L^2(0,1)} \frac{\|Ax\|_{\ell^2}}{\|x\|_{L^2(0,1)}}=\sqrt{\pi}$. The adjoint operator
$A^*$ to $A$ attains the form
\begin{equation} \label{eq:Astar}
[A^*\,y](t):=\sum \limits _{j=1}^\infty y_j\, t^{j-1} \qquad (0 \le t \le 1),
\end{equation}
and is hence for all $y=(y_1,y_2,...) \in \ell^2$ well-defined and injective, which means that $A^*:\ell^2 \to L^2(0,1)$ and that the ranges $\mathcal{R}(A)$ and $\mathcal{R}(A^*)$ are dense in $\ell^2$ and $L^2(0,1)$, respectively.
\end{proposition}

We furthermore consider the operators $A A^*$ and $A^*A$ and
relate the first one to the well-known Hilbert (infinite) matrix  $\mathcal{H}=(\mathcal{H}_{i,j})_{i,j=1}^{\infty}$ with entries
\begin{equation}\label{defHilbert}
\mathcal{H}_{i,j} = \left(\frac{1}{i + j -1}\right).
\end{equation}
Note that this object can also be considered as a bounded linear operator
$\mathcal{H}: \ell^2 \to \ell^2$.

\begin{proposition}
The operator $A A^*: \ell^2 \to \ell^2$ can be represented
by the (infinite) Hilbert matrix $\mathcal{H}$ defined by formula \eqref{defHilbert} as
as
\[ A A^* (y_i)_{j=1}^\infty \to   (\mathcal{H} y)_{i=1}^{\infty}
= (\sum_{j=1}^\infty \mathcal{H}_{i,j} y_j)_{i=1}^{\infty}, \]
i.e., for short we can write
\begin{equation} \label{eq:AH}
 A A^* \,=\, \mathcal{H}.
\end{equation}

The operator $A^*A: L^2(0,1) \to L^2(0,1)$ can be represented as
singular integral operator
\begin{equation}\label{eq:AAast} A^*A: x \to \int_0^1 k(s,t) x(t) dt   \end{equation}
with
\[ \int_0^1 k(s,t) x(t) dt := \lim_{\epsilon \to 0}
\int_0^{1-\epsilon}  \frac{1}{1 - s t} x(t) dt. \]
\end{proposition}
\begin{proof}
From \eqref{eq:A} and \eqref{eq:Astar} it follows, for
$y = (y_i)_{i=1}^\infty \in \ell^2$, directly that
\begin{align*}
AA^* y = \sum_{j=1}^\infty \int_0^1 t^{i-1} t^{j-1} dt  y_j
=  \sum_{j=1}^\infty \frac{1}{i+j-1}  y_j = \mathcal{H} y.
\end{align*}

In the case of the converse composition $A^*A$, we take for $x\in L^2(0,1)$ a cut-off at $t = 1$
and define $x_\epsilon(t) = x(t)\chi_{0,1-\epsilon}(t)$
with the characteristic function $\chi$. Then
it follows that $\|x-x_\epsilon\|_{L^2(0,1)}\to 0$ as $\epsilon \to 0$,
and by continuity we have that $Ax = \lim_{\epsilon\to 0} A x_\epsilon$
in the $\ell^2$-norm.
Thus,
\[  A^*A x = \lim_{\epsilon \to 0} A^*A x_\epsilon \]
with
\begin{align*}
A^* A x_\epsilon =
 \sum_{j=1}^\infty s^{j-1} \int_0^{1-\epsilon}
 t^{j-1} x(t) ds  =
 \int_0^{1-\epsilon}  \sum_{j=1}^\infty (st)^{j-1} x(t) dt  =
  \int_0^{1-\epsilon}  \frac{1}{1 - s t} x(t) dt
\end{align*}
\end{proof}

By using Legendre polynomials, we may derive an  $\Lo\, Q$-decomposition of the
moment operator, i.e., similar
as for matrices a decomposition into a
product of a lower left triangular operator
and an orthogonal operator. For this, we define the
orthonormal Legendre Polynomials on the interval [0,1]:
\[ L_n(x) := \sqrt{2n+1} P_n(2 x-1), \]
where $P_n$ are the standard orthogonal Legendre polynomials on
$[-1,1]$; cf.~\cite{AbSt64}.
The normalizing factor $\sqrt{2n+1}$ makes
$L_n(x)$  an orthogonal basis of $L^2(0,1)$.

\begin{proposition}
The HMP-operator $A$ has the decompostion
\begin{equation} \label{eq:LQ}
 A = \Lo \, Q,
\end{equation}
where $Q$ is the isometry
\[ Q: L^2(0,1) \to \ell^2 \qquad x \to \left(\langle x,L_{n-1}
\rangle_{L^2(0,1)}\right)_{n=1}^\infty \]
and $\Lo$ is a lower triangular operator
\[ \Lo: \ell^2 \to \ell^2 \qquad (y_j)_{i=1}^\infty \to (\sum_{i=1}^\infty  \Lo_{i,j} y_j)_{i=1}^\infty \]
represented by the lower triangular (infinite) matrix
\begin{equation}\label{low}
\Lo_{i,j} =
\frac{\sqrt{(2(j-1)+1}}{(i-1) + (j-1) +1}
\frac{\begin{pmatrix} i-1 \\ j-1 \end{pmatrix}}{\begin{pmatrix} i-1 +j-1 \\ j-1 \end{pmatrix}}, \qquad i,j=1,2,\ldots
\end{equation}
\end{proposition}

\begin{remark} \label{rem:AH}
{\rm It follows from the formula \eqref{low}, which was verified with Mathematica, that $\Lo_{i,j} = 0$ for $j>i$, hence the
matrix has lower triangular shape.
Note that in formula \eqref{low}, the indices start from $i,j=1$, which
is the reason why it differs to the formula (20j) in \cite{Tal87}.
The different sign factor $(-1)^j$  in \cite{Tal87} arises because
of a different (signed) normalization of the polynomials.
Evidently, we have from \eqref{eq:AH} and \eqref{eq:LQ} that
\begin{equation} \label{eq:HCholesky}
\mathcal{H}\,=\, \Lo \,\Lo^*\,,
\end{equation}
where $\Lo^*$ is the upper triangular (infinite) matrix transposed to $\Lo$.
Hence, \eqref{eq:HCholesky} expresses a Cholesky decomposition of the Hilbert matrix $\mathcal{H}$.
}\end{remark}

\section{Proving ill-posedness of type~I in the sense of Nashed} \label{sec:proof}

In order to show  ill-posedness of type~I in the sense of Nashed, we will prove that HMP is ill-posed (i.e.~$\mathcal{R}(A)\not= \overline{\mathcal{R}(A)}^{\,\ell^2}$ and $A^{-1}:\mathcal{R}(A) \subset \ell^2 \to L^2(0,1)$ is unbounded) and that $A$ with $\overline{\mathcal{R}(A)}^{\,\ell^2}\!=\ell^2$, which implies condition \eqref{eq:infdim}, is not compact, see also \cite[pp.91--93]{Hof99} and \cite[p.47]{Angbook02}.

\begin{proposition} \label{pro:typeI}
The operator equation \eqref{eq:opeq} with the operator $A$ from~\eqref{eq:A} is ill-posed of type~I in the sense of Definition~\ref{def:Nashed}.
\end{proposition}

{\parindent0em Proof:} First we consider an (obviously existing) infinite sequence $\{e_i\}_{i=1}^\infty$ in form of an orthonormal system in $L^2(0,1)$ such that $\|e_i\|_{L^\infty(0,1)} \le C$ for some constant $0<C<\infty$ and all $i \in \mathbb{N}$.
Then we have weak convergence $e_i \rightharpoonup 0$ in $L^2(0,1)$ as $i \to \infty$. On the other hand, we have
$$\|Ae_i\|^2_{\ell^2} = \sum \limits_{j=1}^\infty \left(\int \limits_0^1 t^{j-1} e_i(t) dt \right)^2  \le C^2\,\sum \limits_{j=1}^\infty \left(\int \limits_0^1 t^{j-1} dt \right)^2=C^2\,\sum \limits_{j=1}^\infty
\frac{1}{j^2}=\frac{C^2\,\pi^2}{6},$$
and with $e_i \rightharpoonup 0$ for all $j=1,2,...$ also  $\int_0^1 t^{j-1} e_i(t) dt \to 0$ as $i \to \infty$, because all polynomials $t^{j-1}$ represent $L^2(0,1)$-elements. Then we can exchange summation and limitation due to
Lebesgue's dominated convergence theorem, and we thus obtain
$$\lim \limits_{i \to \infty}\|Ae_i\|^2_{\ell^2}= \sum \limits_{j=1}^\infty \lim \limits_{i \to \infty} \left(\int \limits_0^1 t^{j-1} e_i(t) dt \right)^2=0.   $$
This, however, contradicts the boundedness of $A^{-1}$, because there is no constant \linebreak $0<K<\infty$ such that $1=\|e_i\|_{L^2(0,1)} \le K\,\|Ae_i\|_{\ell^2}$ for all $i \in \mathbb{N}$. Hence the Hausdorff moment problem from Definition~\ref{def:HMP} is ill-posed.

\smallskip

{\parindent0em To} prove non-compactness of $A$ we should exploit here a sequence $\{x_i\}_{i=1}^\infty$ in $L^2(0,1)$ which is not uniformly bounded in $L^\infty(0,1)$ but in $L^2(0,1)$. Following an idea of A.~Neubauer we use for this purpose $x_i(t)=\sqrt{i}\,t^i\;(0 \le t \le 1)$ with $\|x_i\|_{L^2(0,1)}\leq \frac{1}{\sqrt{2}}$ for all $i=1,2,\dots$. The image $Ax_i$ converges to zero component-wise, since
\begin{equation}\label{eq:weakconv}
[Ax_i]_j=\int\limits_0^1 t^{j-1}\sqrt{i}t^i\, dt=\frac{\sqrt{i}}{i+j}\rightarrow 0 \mbox{ as } i\rightarrow \infty.
\end{equation}
Whenever this sequence has a subsequence which is norm convergent in $\ell^2$, then the limit must be the zero sequence. However, we find that
$$\|Ax_i\|^2_{\ell^2} = \sum \limits_{j=1}^\infty \left(\int \limits_0^1 t^{j-1} \sqrt{i}\,t^i dt \right)^2= \sum \limits_{j=1}^\infty \left( \frac{\sqrt{i}}{i+j} \right)^2$$
$$= i\,\sum \limits_{j=i+1}^\infty \frac{1}{j^2} \ge i\, \int \limits_{i+1}^\infty \frac{1}{t^2}dt = \frac{i}{i+1} \to 1 \not = 0$$
as $i \to \infty$. Combining this with \eqref{eq:weakconv} implies that $\{Ax_i\}_{i=1}^\infty$ cannot have a norm convergent subsequence, and thus $A$ fails to be compact.
\qed

\begin{remark} \label{rem:non-c}
{\rm Due to Proposition~\ref{pro:typeI} the operator $A: L^2(0,1) \to \ell^2$ from \eqref{eq:A} is non-compact and so is the non-negative self-adjoint operator $\mathcal{H}=AA^*: \ell^2 \to \ell^2$, which means that zero is an accumulation point of the spectrum of the bounded linear operator $\mathcal{H}$ mapping in $\ell^2$.  Hence, the inverse operator $\mathcal{H}^{-1}$
must be unbounded. Moreover, we have $\mathcal{H}^{-1}=(\Lo^*)^{-1}\, \Lo^{-1}$  because of \eqref{eq:HCholesky}.
}\end{remark}

\smallskip
Since the HMP is ill-posed of type I, it is of interest to characterize the spaces $X_1$ and $Y_1$ from Corollary~\ref{cor:equiv} with respect to common properties of its elements. Using the results of the previous sections we immediately find the following conditions. From \eqref{eq:t1alhpa} it follows that $\|x\|_{L^2[0,1]}^2\leq C_1^2 \|Ax\|_{\ell^2}^2$ for all $x\in X_1$. Together with $\|Ax\|_{\ell^2}^2=\langle A^\ast A x,x\rangle_{L^2[0,1]\times L^2[0,1]}$ and \eqref{eq:AAast}, this yields that $x\in X_1$, if there is $C>0$ such that
\[
\int\limits_0^1 x^2(t)\,dt\leq C \int\limits_0^1 \int\limits_0^1 \frac{1}{1-st}x(s)x(t)\,ds\,dt
\]
On the other hand, by a similar argument it follows from \eqref{eq:t1beta} and \eqref{eq:AH} that $y\in Y_1$ if there is $C>0$ such that
\[
\|y\|_{\ell^2}^2\leq C\langle \mathcal{H}y,y\rangle_{\ell^2\times\ell^2}=C\|Ly\|_{\ell^2}^2.
\]

For a further characterization of $Y_1$ we may associate to sequences in $\ell^2$ the associate power series function:
\[ f_y(t):= A^*y = \sum_{j=1}^\infty y_j t^{j-1}, \]
and due to $\|A^*y\|_{L^2(0,1)} = \|f_y\|_{L^2(0,1)}$ and \eqref{eq:t1beta} we conclude that there exists a
subspace  $Y_1 \subset \ell^2$
with
\begin{equation} \label{eq:ineq}
  \|y\|_{\ell^2}  \leq D_1 \|f_y\|_{L^2(0,1)}.  \end{equation}
Since the opposite inequality of \eqref{eq:ineq} holds, as
$A^*$ is bounded,
we have  on $Y_1$ the norm equivalence
$\|f_y\|_{L^2(0,1)}  \sim   \|y\|_{\ell^2}$.

It follows that $f_y$ is convergent at least on a dense set and, since it is
a power series, its convergence radius must be larger or equal to $1$:
$r \geq 1$.  Thus,
we may extend it analytically to the unit disk $\mathbb{D}$. The corresponding
$\ell^2$-norm on the coefficients is the Hardy space $\mathbb{H}^2(\mathbb{D})$:
\[\mathbb{H}^2(\mathbb{D}): = \{f = \sum_{i=1}^\infty c_i z^{i-1}\text{ analytic in }  \mathbb{D}\,:\,
\|c_i\|_{\ell^2} < \infty \},\quad \|f\|_{\mathbb{H}^2(\mathbb{D})} = \|(c_i)_i\|_{\ell^2}. \]
We can associate  $\|c_i\|_{\ell^2}$ with the norm of the Fourier series
(note our index shift convention)
\[  \|y\|_{\ell^2} = \|\sum_{i=0}^\infty  y_{i+1} e^{i n s}\|_{L^2(0,1)} =
\| f_y(e^{i s})\|_{L^2(0,1)}, \]
where $f_y(e^{i s})$ is considered as a function of $s\in (0,2 \pi)$.
Hence,  in $Y_1$,
the $L^2$-norm of the analytic extension to  the unit circle is equivalent to the
$L^2$-norm on the interval $[0,1]$:
\[ \| f_y(e^{i s})\|_{L^2(0,1)}  \leq D_1 \|f_y(t)\|_{L^2(0,1)}.  \]

Thus if we denote by $f_y$ also its analytic extension  to the unit disc, we have
an alternative characterization of the space $Y_1$ as
\[  \|f_y\|_{\mathbb{H}^2(\mathbb{D})} \leq D_1 \|f_y\|_{L^2(0,1)}. \]
Conversely, if a function in $\mathbb{H}^2(\mathbb{D})$ satisfies this inequality
and if its imaginary part vanishes on $[0,1]$, its restriction to $[0,1]$
is in $Y_1$.
Below in Section 6 we will point out some further characterizations of $Y_1$. Unfortunately, none of them allows to derive a simple description of the elements of $Y_1$ in terms of, for example, the decay of its elements. However, it is clear that the set $Y_1$ is nonempty. For instance any finite sequence $(y_i)$ is obviously
in $Y_1$ and thus any polynomial is in $X_1$, although with a
constant that in general grows with the degree. Further, it includes functions that are (almost) singular, and the decay rate of the entries of $(y_i)$ appears to be irrelevant, as we will show in the following.
Define $f_Y = g_\alpha$, with
\[ g_\alpha(t) = (1- t)^\alpha \qquad  \alpha \in (-\frac{1}{2}, 0), \]
Then
\[ \|g_\alpha\|_{L^2(0,1)}^2  = \frac{1}{1+2 \alpha}. \]
The associated coefficient sequence is given by the binomial series
\[   (1- t)^\alpha = \sum_{k=0}^\infty \binomial{\alpha}{k} (-t)^k, \]
hence the sequence of coefficients is
\begin{equation}\label{eq:y} y_i = \binomial{\alpha}{i-1} (-1)^{i-1} \qquad
 \|g_\alpha\|_{\mathbb{H}^2(\mathbb{D})}^2 =
\|y_i\|_{\ell^2}^2 =
\sum_{k=0}^\infty  \binomial{\alpha}{k}^2 = \frac{\Gamma(1+ 2\alpha)}{\Gamma(1+\alpha)^2}.\end{equation}
Thus we have that the constant in \eqref{eq:ineq} is
\[ D_1^2 = \frac{\Gamma(1+ 2\alpha)(1-2 \alpha)}{\Gamma(1+\alpha)^2}.  \]
Although the constant $D_1$ explodes as $\alpha$ approaches $-\frac{1}{2}$. we may choose an infinite sequence $\{\alpha_i\}_{i=1}^\infty$ with elements
in $(-\frac{1}{2},0)$ that yields an infinite sequence of coefficients
$\{y_i\}_{i=1}^\infty$, all linear independent, which are in the stable subspace $Y_1$.
This  verifies non-compactness
of $A^*$
and thus ill-posedness of type I, and hence
by Corollary~\ref{cor:equiv}
also for $A$.

It is interesting to further investigate the sequence $\{y_i\}_{i=1}^\infty$ from \eqref{eq:y}. From Eulers definition of the $\Gamma$-function one show that $c_1 \frac{1}{k^{1+\alpha}}\leq \left|\binomial{\alpha}{k}\right|\leq c_2 \frac{1}{k^{1+\alpha}}$ with positive constants $c_1,c_2$. Further, for $-1<\alpha<0$ we have the identity $$\binomial{\alpha}{k}=\frac{\Gamma(\alpha+1)}{\Gamma(k+1)\Gamma(\alpha-k+1)}=(-1)^{k}\left|\frac{\Gamma(\alpha+1)}{\Gamma(k+1)\Gamma(\alpha-k+1)}\right|, \quad k=0,1,\dots$$ This means that the sequence $(y_i)$ in \eqref{eq:y} is not alternating but positive for all $i\in \mathbb{N}$. In summary, $$c_1 \frac{1}{k^{1+\alpha}}\leq y_i\leq c_2 \frac{1}{k^{1+\alpha}},$$ i.e., the coefficients fall strictly monotonously and slowly, in particular for $\alpha$ close to $-\frac{1}{2}$.

\section{The truncated Hausdorff moment problem and the associated semi-discrete forward operator}
\label{sec:truncated}

In practice, the moment observation is limited to a finite number $n$ of moments, which motivates the replacement of $A$ by a semi-discrete operator $A_n$ possessing a finite dimensional range.
For the discussion of such problem we refer, for example, to the series of publications  in \cite{Angetal99,Askey82,Borwein91,Frontini11,Inglese89,Inglese92,Inglese95,Tagliani99,Tal87,Zellinger21}).

\begin{definition}[truncated Hausdorff moment problem] \label{def:tHMP}
We call the inverse problem of solving the linear operator equation
\begin{equation} \label{eq:opeqn}
A_n\, x\,=\,P_n \,y
\end{equation}
the truncated or finite Hausdorff moment problem if the forward operator $A_n$ maps from the separable Hilbert space $X=L^2(0,1)$ into the
separable Hilbert space $Y=\ell^2$  and attains the form
\begin{equation} \label{eq:An}
[A_n\,x]_j:=\int \limits _0^1 t^{j-1}\,x(t)\,dt \qquad (j=1,2,...,n), \quad [A_n\,x]_j:=0 \quad (j=n+1,n+2,...).
\end{equation}
Precisely, a real function $x$ with support on $[0,1]$ has to be recovered from the finite sequence of its first $n$ moments.
\end{definition}

\begin{remark}\label{remark:indexshift}
{\rm Note that, compared to other papers, including \cite{Tal87}, we have shifted the index $i$ in \eqref{eq:A}, \eqref{eq:An} from starting at $i=0$ to starting at $i=1$. In the truncated version we sum up to $n=j-1$ instead of $n=j$ as in \cite{Tal87}. Therefore we will have an index shift $n\leftarrow n+1$ when citing the results of \cite{Tal87}.
}\end{remark}

The truncated Hausdorff moment problem is strongly underdetermined, because a real function over $[0,1]$ has to be identified from an $n$-dimensional vector. Evidently, we have $A_n=P_n A$, where $P_n$ is the orthogonal projector in $\ell^2$ on the $n$-dimensional subspace with zeros in all components with numbers greater than $n$. Semi-discrete operators of this general structure have early been discussed in the paper \cite{Bertero85}. The following proposition with
properties of $A_n$ is easy to prove.

\begin{proposition} \label{pro:semi}
For the operator $A_n: L^2(0,1) \to \ell^2$, which was introduced by formula~\eqref{eq:An},  we have the following properties:
$A_n$ is a bounded but non-injective linear operator with $n$-dimensional range $\mathcal{R}(A_n)$. Consequently, $A_n$ is a compact operator. The pseudoinverse $A_n^\dagger: \ell^2 \to L^2(0,1)$ and the
adjoint operator $A_n^*$ to $A_n$, which attains the form
\begin{equation} \label{eq:Anstar}
[A_n^*\,y](t):=\sum \limits _{j=1}^n y_j\, t^{j-1} \qquad (0 \le t \le 1),
\end{equation}
are bounded linear operators and well-defined everywhere on $\ell^2$. Both ranges $\mathcal{R}(A_n^\dagger)$ and $\mathcal{R}(A_n^*)$ coincide with the $n$-dimensional space ${\rm span}(1,t,...,t^{n-1})$
of polynomials up to degree $n-1$. Consequently, the infinite-dimensional null-space $\mathcal{N}(A)$ contains all functions in $L^2(0,1)$ which are orthogonal to all polynomials  up to degree $n-1$.
We have pointwise convergence $\|Ax-A_n x\|_{\ell^2} \to 0$ as $n \to \infty$ for all $x \in L^2(0,1)$.
\end{proposition}

The results from Proposition~\ref{pro:semi}, in particular the conditions $\mathcal{R}(A_n) = \overline{\mathcal{R}(A_n)}^{\ell^2}$ and ${\rm dim}(\mathcal{R}(A_n))=n$, which indicate that $A_n: L^2(0,1) \to \ell^2$ is non-injective and non-surjective with a continuous pseudoinverse $A_n^\dagger$, is the basis for the following proposition. The last assertion of Proposition~\ref{pro:An} exploiting the inverse of the Hilbert matrix
is taken from the proof of Theorem~1 in \cite{Tal87}.

\begin{proposition} \label{pro:An}
The operator equation \eqref{eq:opeqn} with the operator $A_n$ from~\eqref{eq:An} is well-posed in the sense of Nashed (cf.~Definition~\ref{def:Nashed}), but ill-posed in the sense of Hadamard. For  all $y \in \ell^2$ the uniquely determined minimum-norm solution $x^\dagger_n$ to equation~\eqref{eq:opeqn} exists and can be verified as $x^\dagger_n=A^\dagger_n y=A^\dagger_n P_n y$. In detail, we have for the vector $\underline v=([P_n y]_1,...,[P_n y]_n) \in \mathbb{R}^n$  with $\|P_n y\|_{\ell^2}=\|\underline v\|_2$ the equation
\begin{equation} \label{eq:vector}
\|A_n^\dagger P_n y\|^2_{L^2(0,1)} =\langle \mathcal{H}_n^{-1} \underline v, \underline v\rangle_2,
\end{equation}
where $\mathcal{H}_n \in \mathbb{R}^{n \times n}$ is the corresponding ill-conditioned
$n$-dimensional segment of the Hilbert matrix \eqref{defHilbert}, and $\|\cdot\|_2$ as well as $\langle \cdot,\cdot\rangle_2$ denote the Euclidean norm and Euclidean inner product in $\mathbb{R}^n$, respectively.
\end{proposition}

 \begin{remark} \label{rem:rem3}
{\rm We have $\|A-A_n\|_{\mathcal{L}(L^2(0,1),\ell^2)} \not\to 0$, because the sequence $\{A_n\}_{n=1}^\infty$ of compact operators cannot converge in norm to the non-compact operator $A$.
Instead of the ill-posedness of the `continuous' equation \eqref{eq:opeq} we have in general ill-conditioning  of the `semi-discrete' equation \eqref{eq:opeqn} even if $n$ is moderate, since, as outlined in the
subsequent section, the operator norm $\|A_n^\dagger\|_{\mathcal{L}(\ell^2,L^2(0,1))}$ tends to grow very fast with $n$. From \eqref{eq:vector},
in combination with the classical result
\begin{equation}\label{eq:hilbertmatrx_norm}
\|\mathcal{H}_n^{-1} \|_2  \le \hat C\,\exp\left( 4\ln(1+\sqrt{2}) \right)\le \hat C\,\exp(3.526\, n)
\end{equation}
from \cite{Todd54} and \cite{Wilf70} concerning the spectral norm of the inverse of the Hilbert matrix, we can estimate
with the constant $\hat C>0$  independent of $n$ as
$$ \|A_n^\dagger y\|^2_{L^2(0,1)}=\|A_n^\dagger P_n y\|^2_{L^2(0,1)} \le \hat C\,\exp(3.526\, n) \|P_n y\|^2_{\ell^2}\le \hat C \,\exp(3.526\, n) \|y\|^2_{\ell^2}$$
and consequently obtain
\begin{equation}\label{eq:Andagger_bound}
\|A_n^\dagger\|_{\mathcal{L}(\ell^2,L^2(0,1))} \le \sqrt{\hat C}\, \exp(1.763\,n).
\end{equation}
Numerical illustrations concerning the upper estimate \eqref{eq:Andagger_bound} can be found in Section~\ref{sec:casestudies}.
}
\end{remark}

\section{Error estimates for the truncated problem under noisy data} \label{sec:errest}

Let for the not available element $y \in \mathcal{R}(A) \subset \ell^2$ denote by $\tilde x \in L^2(0,1)$ the corresponding (uniquely determined) solution to equation \eqref{eq:opeq}. Moreover, let $y^\delta \in \ell^2$ be an available perturbation to $y$ satisfying the noise model
\begin{equation} \label{eq:noise}
\|y^\delta-y\|_{\ell^2} \le \delta,
\end{equation}
with noise level $\delta>0$. Then we have with $x_n^\delta:=A_n^\dagger y^\delta$ and  $x^\dagger_n:=A_n^\dagger y$ and
due to the orthogonality of the range of $A_n^\dagger$ and of the null-space of $A_n$ that
$$\|x^\delta_n-\tilde x\|^2_{L^2(0,1)} = \|x^\delta_n-x_n^\dagger\|^2_{L^2(0,1)}+ \|x_n^\dagger - \tilde x\|^2_{L^2(0,1)},$$
and further that
\begin{equation}\label{eq:est1}
\|x_n^\delta -\tilde x\|^2_{L^2(0,1)} \le  \|A_n^\dagger\|^2_{\mathcal{L}(\ell^2,L^2(0,1))} \,\delta^2 + \|(A_n^\dagger A -I) \tilde x\|^2_{L^2(0,1)}.
\end{equation}
The upper estimate \eqref{eq:est1} of the norm square error for the solution $\tilde x$ of the original operator equation \eqref{eq:opeq}
with $A$ from \eqref{eq:A}, by means of using as approximate solutions  the minimum-norm solutions $A_n^\dagger y^\delta$ of equation \eqref{eq:opeqn} based on noisy data $y^\delta$, yields a worst case error bound on the right-hand side with two terms.
In particular, the first term of the error bound expresses the noise amplification with amplification factor $\|A_n^\dagger\|_{\mathcal{L}(\ell^2,L^2(0,1))}$ that tends to grow extremely with $n$, see \eqref{eq:Andagger_bound}.
On the other hand, the second term of the error bound depending on $n$ and on the smoothness of the solution $\tilde x$ tends to decay to zero whenever $n$ tends to infinity as the following proposition indicates.

\begin{proposition} \label{pro:convergence}
We have $\lim \limits_{n \to \infty}\|(A_n^\dagger A -I) \tilde x\|_{L^2(0,1)} =0$ for arbitrary $\tilde x \in L^2(0,1)$.
\end{proposition}

{\parindent0em Proof:} The proof ideas sketched in the following can be found in \cite{Tal87} or \cite{Angbook02}. By $\{L_i\}_{i=0}^\infty$ denote the orthonormal basis in $L^2(0,1)$ formed by the normalized and shifted Legendre polynomials introduced above in {Section~\ref{sec:HMPintro}}. This is the results of the Gram-Schmidt orthogonalization procedure applied to the non-orthogonal basis $\{t^{i-1}\}_{i=1}^\infty$ in $L^2(0,1)$.
 By construction of this basis, we have $\{L_{i-1}\}_{i=1}^n$ as orthonormal basis of the $n$-dimensional subspace ${\rm span}(1,t,...,t^{n-1})$ of $L^2(0,1)$.
Then we can split the solution $\tilde x=\sum_{i=1}^\infty \langle \tilde x,L_{i-1} \rangle_{L^2(0,1)}\,L_{i-1}$ with  $\|\tilde x\|^2_{L^2(0,1)}=\sum_{i=1}^\infty \langle \tilde x,L_{i-1} \rangle_{L^2(0,1)}^2 < \infty$ in a unified manner as $\tilde x=x_1+x_2$ with $x_1=A_n^\dagger A \tilde x \in {\rm span}(L_0,...,L_{n-1})$ and $x_2 \in \mathcal{N}(A_n) \perp {\rm span}(L_0,...,L_{n-1})$. Thus we arrive at
$$\|(A_n^\dagger A -I) \tilde x\|^2_{L^2(0,1)}= \sum \limits_{i=n}^\infty \langle \tilde x,L_i \rangle^2_{L^2(0,1)} \to 0 \quad \mbox{as} \quad n \to \infty.$$
This proves the proposition. \qed

To prove convergence rates for $\|(A_n^\dagger A -I) \tilde x\|_{L^2(0,1)} \to 0$ as $n \to \infty$, additional smoothness conditions on $\tilde x$ have to be imposed. We only mention the following two such results, which are taken from \cite[p.~511]{Tal87} and \cite[Remark~4.1]{Angbook02}.

\begin{proposition} \label{pro:conrates}
If $\tilde x \in H^1(0,1)$, then the estimate
\begin{equation} \label{eq:h1error}
\|(A_n^\dagger A -I) \tilde x\|_{L^2(0,1)} \le \frac{1}{2n}\, \|\tilde x\|_{H^1(0,1)}
\end{equation}
holds true. If we even have $\tilde x \in H^2(0,1)$, then the estimate
$$\|(A_n^\dagger A -I) \tilde x\|_{L^2(0,1)} \le \frac{1}{2 \sqrt{2}\, n^2}\, \|\tilde x\|_{H^2(0,1)}$$
is valid.
\end{proposition}

Error and stability estimates for the original (non-truncated) Hausdorff moment problem  can be found in {Section~\ref{sec:stabest}} below.

\section{Relation to other problems and inversion}\label{sec_relations}
The HMP arises as a simplification in several
classical direct and inverse problems. Moreover,  inversion
formulae for the HMP have been derived, which can be used in other related inverse problems as well. These various ramifications are the topic of this section.

\subsection{The Laplace transform}
At first we
relate the HMP to the Laplace transform.
Recall that the Laplace transform is defined as
\begin{align*}
\mathcal{L} : L^2[0,\infty] &\to  L^2[0,\infty] \\
           f & \to \mathcal{L} f(s):= \int_0^\infty e^{-s t} f(t)\, dt.
\end{align*}
The relation to the HMP operator is the following (see, e.g., \cite{Tal87}):
\begin{proposition}
For $x \in L^2([0,1])$, define $ \tilde{x}$ as
\[ \tilde{x}(\tau):=  x(e^{-\tau}) \qquad \tau \in [0,\infty] \]
Then we have
\[ [A x]_j =  \mathcal{L}(\tilde{x})(j), \qquad j = 1,2,\ldots   \]
\end{proposition}
\begin{proof}
This follows easily by the substitution $t = e^{-\tau}$ with
$dt = -e^{-\tau} d\tau $.
\end{proof}
Thus, the HMP operator is equivalent to a sampling of the
Laplace transform of $\tilde{x}$ at the integer points.
An inversion formula for the Laplace transform due to Widder \cite{Widd41}
is closely related to the classical HMP inversion formula, which
we present in the next section.

\bigskip
\subsection{Hausdorff's range characterization}
There is a classical characterization of the range of $A$ due to Hausdorff
\cite{Hausdorff23}.
For $(y_i)_{i=1}^\infty$,
we define the forward differences
\[ \mu_{m,n}:= \sum_{l=0}^n (-1)^l  \binomial{n}{l}
y_{m+l+1}, \qquad m,n \in \{0,1,\ldots\}, \]
and for $N \in \mathbb{N}$ and $0\leq m \leq N$,
\[ \lambda_{N,m}:=
\binomial{N}{m}
\mu_{m,N-m}
\]
Then, according to Hausdorff, a sequence of moments is in $\mathcal{R}(A)$ if and only if
there exists a constant $L$ such that  for all $N$,
\begin{equation}\label{eq:Haus}
 (N+1) \sum_{m=0}^N \left| \lambda_{N,m}\right|^2  \leq L . \end{equation}
Similar conditions can be formulated for moments of $L^p$ functions or functions of bounded variations.

Note that our previous  $\Lo Q$-decomposition
of the operator $A$ allows for an alternative characterization
of the range; however, it is not obvious,
how Hausdorff's result is related to that, and this
is what we would like to study in this part in more detail.
The range characterization by the
 $\Lo Q$-decomposition is based on the fact that
the operator $\Lo: \ell^2 \to \ell^2$ in the decomposition \eqref{eq:LQ} is triangular,  and hence its inverse can be calculated
by back-substitution.
\begin{lemma}
The inverse to $\Lo$ is given in (infinite) matrix form by
\begin{equation}\label{eq:Linverse}
\Lo^{-1}_{i,j} =
(-1)^{(i-1)+(j-1)} \sqrt{(2(i-1)+1}
{\begin{pmatrix} i-1 \\ j-1 \end{pmatrix}}{\begin{pmatrix} i-1 +j-1 \\ j-1 \end{pmatrix}}, \qquad i,j=1,\ldots
\end{equation}
\end{lemma}
Again this agrees with the formula (20j) in \cite{Tal87} up to the
sign factor arising from a different normalization.
A full characterization of the range $\mathcal{R}(A)$ of  $A: L^2(0,1) \to \ell^2$
in form of a Picard-type condition
is thus
that
\begin{equation}\label{rangecond}  \mathcal{R}(A) = \mathcal{R}(\Lo) =
\left\{ (y_i)_{i=1}^\infty \,:\, \|\Lo^{-1} y\|_{\ell^2} < \infty \right\}.
\end{equation}
It is quite interesting that due to the ill-posedness of type~I,
and Corollary~\ref{cor:equiv},
 this range has to include a closed
infinite-dimensional subspace $Y_1$
with a constant $C>0$ such that
\[ \|A^{-1} y\|_{L^2[0,1]} =\|\Lo^{-1} y\|_{\ell^2} \leq C \|y\|_{\ell^2}  \qquad \forall y \in Y_1 \subset \ell^2. \]

Now,  Hausdorff's condition  \eqref{eq:Haus} is another
characterization, which is related  to
$\Lo$ in a non-obvious way.   By using  the substitution $l \to N-m-l$, we find,
\begin{align*}
\mu_{m,N-m} &=
 \sum_{l=0}^{N-m} (-1)^l \left(\begin{array}{c} N-m \\ l\end{array} \right)
y_{m+l+1}\\
&=
 (-1)^{N-m} \sum_{l=0}^{N-m} (-1)^l \left(\begin{array}{c} N-m \\ N-m-l\end{array} \right)
y_{N-l+1} \\
&=
 (-1)^{N-m} \sum_{l=0}^{N-m} (-1)^l \left(\begin{array}{c} N-m \\ l\end{array} \right)
\mu_{N-l+1},
\end{align*}
where up to a sign, this is
nothing but the $(N-m)$-th backward difference at the index $N+1$.
We
 define the following triangular right upper matrix:
$R_N \in \R^{(N) \times (N)}$
\begin{align*}
 R_N: =
\begin{pmatrix}
 1  &  -\binomial{N-1}{1}
  &
\binomial{N-1}{2} & \ldots &  \binomial{N-1}{N-2} & (-1)^{N-1} \\
0 & 1& \binomial{N-2}{2} & \ldots & & 1 \\
0 & 0 & 1 & \ldots &  \ldots & -1 \\
0 & 0 & 0 & 1 &  \ldots & 1 \\
0 & 0 & 0 & 0 &  1 & -1 \\
0 & 0 & 0 & 0 & 0 &  1
\end{pmatrix},
\end{align*}
i.e., the nonzero elements are
 \begin{align*}
(R_N)_{i,j} =&
 (-1)^{N-i}  (-1)^{N-j} \left(\begin{array}{c} N-i \\ N-j\end{array} \right) =
 (-1)^{N-i}  (-1)^{N-j} \left(\begin{array}{c} N-i \\ j-i\end{array} \right) \\
  \qquad
& \qquad \qquad \qquad \qquad \qquad \qquad \qquad \qquad \qquad \qquad   i = 1, N, j = i,\ldots, N. \end{align*}
Then, replacing in the formula for $\mu_{m,N-m}$ the value $N \to N-1$
and setting $i = m+1$ and $j = N -l$, we find with this definition that
\[
(\mu_{i-1,N-1 - (i-1)})_{i=1}^N
=
\begin{pmatrix} \mu_{0,N-1} \\
\mu_{1,N-2} \\
\ldots  \\
\ldots \\
\mu_{N-1,0}
\end{pmatrix} =
R_N \begin{pmatrix} y_1\\ y_2 \\ \ldots \\ \ldots \\ y_{N} \end{pmatrix}.
\]
Moreover, define the diagonal matrix (setting $i = m+1$):
 \[ D_N:= \sqrt{N}\text{diag}\left(\binomial{N-1}{i-1} \right), \qquad i = 1,\ldots N \]
which allows us to write
\[ \sqrt{N} \lambda_{N-1,i-1} =  (D_N R_N P_N y)_{i=1}^N. \]
Thus,  the Hausdorff condition  \eqref{eq:Haus} is equivalent to
\[\sup_N \|D_N R_N  P_{N} y\|_{\R^N}^2 \leq  L. \qquad y \in \ell^2,\]
where $\|\cdot\|_{\R^N}$ denotes the Euclidean norm in $\R^N$.
On the other hand, the range condition \eqref{rangecond} can be written as
\[ \sup_{N} \|P_N \Lo^{-1}  y\|_{\ell^2}  \leq C, \]
which raises the question of the relation of
$D_N R_N P_N$  and  $P_N\Lo^{-1}$. At least asymptotically they should
generate equivalent norms. The interesting result is the following.

\begin{proposition}
Define the $R^{N\times N}$ matrix
\[ V_N:=  D_N R_N  P_{N} \Lo P_{N}. \]
Then,
$ V_N^T V_N$ is a diagonal matrix with
\[ V_N^T V_N = : T_N = \text{\rm diag}\left(\frac{\binomial{N-1}{k-1}}{\binomial{N-1+k}{k-1}} \right)_{k=1,N}. \]
Moreover, extending $T_N$ by $0$ to $k>N$ yields an operator on $\ell^2$
which converges pointwise to the identity
\[ \lim_{N\to \infty} T_N  = Id. \]
In particular we have for all $y \in \mathcal{R}(A)$
\[  \|D_N R_N  P_{N} y\|_{\R^N}^2  = \|T_N^\frac{1}{2} P_N \Lo^{-1} y\|_{\R^N}.  \]
\end{proposition}
\begin{proof}
The result for the
matrix $V_N^T V_N$ has been calculated and proven  by
Askey, Schoenberg and Scharma \cite{Askey82}. We have verified the result
by symbolic calculation using Mathematica, also to adopt the result to our notation.
It follows that for all $y \in \ell^2$,
\begin{align*}
 \| T_N^\frac{1}{2}  P_N y\|_{\R^N}^2  &=
(T_N  P_N y,P_N y)_{\ell^2} =
(V_N^T V_N P_N y,P_N y)_{\ell^2} \\
&= \|V_N P_N y \|_{\R^N}^2
= \| D_N R_N  P_{N} \Lo P_{N} y\|_{\R^N}^2.
\end{align*}
Set $z = P_N L P_N y$ and note that this relation can be inverted as
$P_N y = P_N \Lo^{-1} P_N z$. (This follows, e.g. from the fact that
$P_N$ is an orthogonal projector, and that\linebreak \mbox{$\Lo^{-1} P_N \Lo P_N = Id$} from the
triangular structure.)
This yields
\begin{align*}
 \| T_N^\frac{1}{2}  P_N \Lo^{-1} P_N z\|_{\R^N}^2
= \| D_N R_N  P_{N} z\|_{\R^N}^2.
\end{align*}
Finally, by the triangular structure, it follows that
 $P_N \Lo^{-1} (I-P_N) = 0$, thus we have that
$\| T_N^\frac{1}{2}  P_N \Lo^{-1} P_N z\|_{\R^N} =
 \| T_N^\frac{1}{2}  P_N \Lo^{-1}  z\|_{\R^N}.$

The fact hat $T_N$ converges pointwise to the identity can
be verified by some elementary calculations: We find that
\[ \frac{\binomial{N-1}{k-1}}{\binomial{N-1+k}{k-1}} =
\Pi_{j=0}^{k-1} \left( 1- 2 \frac{j+1}{j + N + 1} \right), \]
from which we observe that the diagonal entries are montonically decreasing
and that they converge pointwise for $k$ fixed to $1$ as $N \to \infty$.
\end{proof}

\subsection{The linearized radially symmetric impedance tomography problem}
The electrical impedance tomography (EIT) problem is another classical inverse problem, which is related to the HMP problem.
In EIT, the aim is to extract information about the conductivity
from boundary measurements of current/voltage pairs.
Since the definition of the problem in the
seminal paper of Calderon \cite{Ca80}, it has been investigated in various direction
and now serves as the  paradigmatic instance of a parameter identification
problem from boundary measurements; see, e.g., the review \cite{Bo02}.

In a  mathematical formulation, the problem  is to consider
solution of the boundary value problem on a Lipschitz domain $\Omega$
\begin{equation}\label{maineq}
\begin{split}
\mbox{div} (\gamma \nabla u) = 0 \quad \text{in } \Omega, \qquad
u = f  \quad \text{ on } \Omega.
\end{split}
\end{equation}
The data for the problem are multiple or infinitely many
pairs of Cauchy-data \linebreak $(f, \gamma \frac{\partial}{\partial n} u|_{\partial \omega})$
on the boundary, and the interest is to  recover the conductivity
$\gamma(x)$ in the interior $\Omega$. The data can be encoded into the so-called
Dirichlet-to-Neumann (DtN) map,
$\Lambda_\gamma: H^\frac{1}{2}(\partial \Omega) \to
H^{-\frac{1}{2}}(\partial \Omega)$, defined as the mapping
$f \to \gamma \frac{\partial}{\partial n}u |_{\partial \Omega}$, i.e.,
from Dirichlet boundary data to Neumann boundary data. It is convenient
to subtract from the data the corresponding DtN operator of
a constant known background conductivity (which we take here
as $\gamma_0 = 1$) such that
the inverse problem amounts to  reconstructing
a perturbation of the background
$\gamma = 1 + \sigma$ from the perturbation of the DtN operator
$\Lambda_\gamma -\Lambda_{1}$. Finally, for small perturbation
it makes sense to perform a linearization such that
the linearized impedance tomography problem uses the data
$\Lambda_{1}'[\sigma]$ instead of $\Lambda_\gamma -\Lambda_{1}$,
with $\Lambda_{1}'$ denoting the Frechet derivative.

In the simplest case of $\Omega = \{(x,y)\subset \R^2 |x^2+y^2 \leq 1\}$
being the unit disk and if the perturbed conductivity
$\sigma$ is radially symmetric, then the problem is closely related to
HMP.  Indeed, the operator $ \Lambda_{1}'$ can be expressed as
\[ \langle \Lambda_{1}'[\sigma] f, g\rangle
= \int_{\Omega} \sigma (x) \nabla u_f(x).\nabla u_g(x) dx,\]
where $u_f, u_g$ are solutions to \eqref{maineq} with $\gamma = 1$,
i.e., harmonic functions. In case of $\sigma$ being radially symmetric:
\[ \sigma(x)  = \sigma(\sqrt{x^2+y^2}), \]
the only relevant information in $\Lambda_{1}'$ is in the diagonal,
i.e., it suffices to take $f = g$; cf.~\cite{Ki17}.
Furthermore, we may choose a
sequence of orthonormalized funtions $f$ on the boundary
such as $f_n(\phi) \sim \{\sin(n \phi),\cos(n\phi)\}$.
The corresponding solutions to \eqref{maineq} (with the normalization
such that $\|\nabla u_f\|_{L^2(\Omega)} =1$) is then given
in polar coordinates  as
\[ u_{n,s}(r,\phi) = c_n r^n \sin(n \phi), \qquad
u_{n,c}(r,\phi) = c_n r^n \cos(n \phi),  \qquad c_n =\frac{\sqrt{n+1}}{n \sqrt{2\pi}}.\]
Thus, we find that
\[  |u_{n,s}(r,\phi)|^2  =   |u_{n,c}(r,\phi)|^2 = c_n^2 n^2 r^{2(n-1)}, \]
which yields for $n = 1,\ldots$
\begin{align*}
 \langle \Lambda_{1}'[\sigma] f_n, f_n\rangle
 &= \int_{\Omega} \sigma(\sqrt{x^2+y^2}) |\nabla u_n,\{c,s\}(x,y)|^2 dxdy \\
& =  c_n^2 n^2 \int_0^{2\pi} \int_{0^1} \sigma(r) r^{2(n-1)} r drd\phi
=  (n+1) \int_{0}^{1} \sigma(r) r^{2n-1} dr \\
&=
\frac{(n+1)}{2} \int_{0}^{1} \sigma(\sqrt{t}) t^{n-1} dt.
\end{align*}
Thus, we observe that in this radially symmetric case
the impedance tomography problem  essentially agrees --- up to
a diagonal scaling $D = \text{diag}(\frac{(n+1)}{2})_{n} $ --- to the
HMP:
\[  \Lambda_{1}'[\sigma]  = D A \tilde{\sigma}, \]
with $\tilde{\sigma}(t) = \sigma(\sqrt{t})$.
An inversion formula for the linearized impedance tomography problem with forward operator $\Lambda_{1}'$, which strongly resembles  the formula for $\Lo^{-1}$, has been stated in~\cite{Ki17}.

\section{Stability estimates for the Hausdorff moment\\ problem and associated moduli of continuity } \label{sec:stabest}

\subsection{Conditional stability estimates for bounded Sobolev-norms}

Since the HMP is ill-posed, only conditional stability estimates
may be expected. That is, we have to restrict the solution
to a compact set. Even if this is imposed, one can only
expect for classical regularity sets at most logarithmic
stability estimates due to the exponential ill-posedness of the
problem. For instance, from results of \cite{Tal87}, such stability
results follow.

For the next results, we define the data norm
\begin{align}\label{eq:defdel}  \delta = \|A x\|_{\ell^2} =
\left[ \sum_{j=1}^\infty \left( \int_{0}^1 x(t) t^{j-1} dt \right)^2  \right]^\frac{1}{2}. \end{align}
For conditional stability estimates, one is interested
in bounding the modulus of continuity
\begin{equation} \label{eq:modcont}
\omega_M(\delta):=\sup_{x \in M,\, \|Ax\|_{\ell^2} \le \delta\,} \|x\|_{L^2(0,1)} \leq  \psi(\delta),
\end{equation}
where $M$ is an appropriate compact subset of $L^2(0,1)$ and $\psi$ an index function that characterizes a rate. Then, $\omega_M(\delta)$ is increasing in $\delta>0$ with the limit condition
$\lim_{\delta \to 0\,} \omega_M(\delta)=0$. Note that, for constants $\lambda>1$ and
centrally symmetric and convex sets $M$, we have $\omega_{\lambda M}(\delta)=\lambda\,\omega_M(\delta/\lambda)$.
For further details of this concept we refer, for example, to \cite{HMS08}.

More general, one may replace  $\|x\|_{L^2(0,1)}$ in \eqref{eq:modcont} by some
alternative norm or by $|\ell(x)|$, with $\ell$ being
a linear functional (cf.~{Subsection~\ref{sub:t1}} below).

\medskip

We start, however, with \eqref{eq:modcont} and consider $H^1$-bounds and the compact set \linebreak $M=\{x \in L^2(0,1): \|x\|_{H^1(0,1)} \le E\}$ in~\eqref{eq:modcont}.

\begin{theorem} \label{thm:log}
Assume that we have the  a priori bound
\[ \|x\|_{H^1(0,1)} \leq E\,. \]
Then we obtain, with $\delta$ defined in \eqref{eq:defdel} and for sufficiently small $\frac{\delta}{E}$, the conditional stability estimate
\[ \|x\|_{L^2(0,1)} \leq \frac{7}{\sqrt{8}}E\left(-\ln\left(C\frac{\delta}{E}\right)\right)^{-1}\,,  \]
where the constant $C>0$ is independent of $E$ and $\delta$. Consequently, the asymptotics for the corresponding modulus of continuity is given as
$$ \sup_{x \in L^2(0,1):\, \|x\|_{H^1(0,1)} \le 1, \;\|Ax\|_{\ell^2} \le \delta\,} \|x\|_{L^2(0,1)} = \mathcal{O} \left[ \frac{1}{\ln(\frac{1}{\delta})}\right] \quad \mbox{as} \quad \delta \to 0.$$
\end{theorem}
\begin{proof}
As in \cite{Tal87}, we may represent $x$ in terms of the Legendre Polynomials with
\[ x = \sum_{i=1}^\infty  \lambda_i L_{i-1}(t) \qquad
\lambda_i = \Lo^{-1} (A x).   \]
Following \cite{Tal87}, we split the norm into two parts. While not needed for this proof, we adapt our index shift from Remark \ref{remark:indexshift} compared to \cite{Tal87}.
Let $x_{N-1}$ be the projection onto the span of the first $N-1$ Legendre Polynomials, and
$r_{N-1}$ be the remainder. Because of orthogonality, we have
\[ \|x\|_{L^2(0,1)}^2  = \|x_{N-1}\|_{L^2(0,1)}^2  +  \|r_{N-1}\|_{L^2(0,1)}^2.  \]
For the tail, we have the approximation properties of orthogonal
polynomials, cf.~ \cite[Eq.(28)]{Tal87} that
\[ \|r_{N-1}\|_{L^2(0,1)}^2 \leq \frac{E^2}{4 N^2} \]
For the corresponding projected part, it
 follows by orthogonality that with $P_N$ being the projector onto
 the first $N$ coefficients in $\ell^2$ that
\[ \|x_{N-1}\|_{L^2(0,1)}= \|P_{N-1} \lambda\|_{\ell^2} = \| P_{N-1} \Lo^{-1} P_{N-1}  (A x)\|_{\ell^2}. \]
Here we used the lower triangular structure, i.e., the first $N-1$ coefficients
can be calculated from the first $N-1$ moments.
This yields that
\[ \|x\|_{L^2(0,1)}^2
\leq  \frac{E}{4 N^2} + \|P_{N-1} \Lo^{-1} P_{N-1}\|_{\mathcal{L}(\ell^2)}^2 \|A x\|_{\ell^2}^2 \]
The norm of $\|\Lo^{-1} P_{N-1}\|_{\mathcal{L}(\ell^2)}^2$ has been bounded in  \cite{Tal87} such that we have by using the (infinite) Hilbert matrix $\mathcal{H}$ and by recalling Remark~\ref{rem:non-c}
\begin{align*}   \|P_{N-1} \Lo^{-1} P_{N-1}\|_{\mathcal{L}(\ell^2)}^2 \leq
 \|\Lo^{-1} P_{N-1}\|_{\mathcal{L}(\ell^2)}^2 &=
 \lambda_{max}(P_{N-1} (\Lo^*)^{-1} \Lo^{-1} P_{N-1}) \\
 &=
\lambda_{max}(P_{N-1} \mathcal{H}^{-1} P_{N-1}) \leq  \hat C \exp(3.5\,N),
\end{align*}
where $\lambda_{max}$ denotes the largest eigenvalue of the corresponding self-adjoint operator mapping in $\ell^2$ and where we have simplified the multiplier $3.526$ in the exponent to $3.5$.
Thus, we end up with
\begin{equation}\label{eq:estimate_step}
\|x\|_{L^2(0,1)}^2 \leq  \frac{E^2}{4 N^2}  + \hat C\,\exp(3.5 \, N) \, \delta^2.
\end{equation}
By balancing these two terms we find  $N$ as the solution of
\[\frac{1}{4\hat{C}} \frac{E^2}{\delta^2}  = N^2\,\exp(3.5\, N),\]
which is given by
\[N=\frac{4}{7} W\left(\pm \frac{7}{8\sqrt{\hat{C}}}\frac{E}{\delta}\right) \]
where $W$ is the principal branch of the Lambert-W function \cite{lambertw}, defined as
\begin{equation}\label{eq:lambertw}
z=W(z)e^{W(z)}.
\end{equation}
Inserting the expression for $N$ into \eqref{eq:estimate_step}  yields
\begin{equation}\label{eq:estimate_step2}
\|x\|_{L^2(0,1)}^2\leq 2\hat{C}\delta^2 \exp\left(2 W\left(\frac{7}{8\sqrt{\hat{C}}}\frac{E}{\delta}\right)\right)= \frac{49}{32}\frac{E^2}{W\left(\frac{7}{8\sqrt{\hat{C}}}\frac{E}{\delta}\right)^2}.
\end{equation}
For $z\rightarrow \infty$ in \eqref{eq:lambertw}, i.e., $\delta\rightarrow 0$ above, we have the asymptotical expansion
\[
W(z)=\ln z-\ln\ln z +o(1),
\]
see \cite{lambertw}. This yields $W(z)\geq K \ln z$ for any $0<K<1$ and $z=z(K)$ large enough. Without loss of generality we set $K=\frac{1}{2}$. Combining this with \eqref{eq:estimate_step2} and taking the square root gives
\[
\|x\|_{L^2(0,1)}\leq \frac{7}{\sqrt{8}}E\frac{1}{\ln\left(\frac{7}{8\sqrt{\hat{C}}}\frac{E}{\delta}\right)}.
\]
Rearranging completes the proof.
\end{proof}

\smallskip

Furthermore, we may verify that the logarithmic conditional stability result from Theorem~\ref{thm:log} cannot be improved to
H\"older-type conditional stability rates. This is done by a counterexample in the following proposition, where even the extended
Sobolev space situation of $M=\{x \in L^2(0,1): \|x\|_{H^k(0,1)} \le 1\}\;(k=1,2,...)$  with respect to the modulus of continuity $\omega_M(\delta)$ is exploited.
\begin{proposition} \label{pro:mu}
For any $\mu \in (0,1)$, any constant $C>0$, and any integer
$k\geq 1$,
there exists functions $x$ with
\[ \|x\|_{H^k(0,1)} \leq 1 \]
such that
\[\|x\|_{L^2(0,1)} \geq C \|A x\|_{\ell^2}^\mu. \]
 Consequently, for all $k=1,2,...$, an asymptotic bound  of H\"older-type
$$ \sup_{x \in L^2(0,1):\, \|x\|_{H^k(0,1)} \le 1,\;\|Ax\|_{\ell^2} \le \delta} \|x\|_{L^2(0,1)} = \mathcal{O} \left(\delta^\mu\right) \quad \mbox{as} \quad \delta \to 0$$
cannot even hold for arbitrarily small exponents $\mu>0$.
\end{proposition}
\begin{proof}
Take a fixed function $g \in C_0^\infty([0,1])$ with
$g \not = 0$, compact support in $[0,1]$ and with the first $m$ moments vanishing:
\[ \int_0^1 g(t) t^{j-1} dt = 0 \qquad j = 1,\ldots, m. \]
For instance, $g$ could be taken as the $m$-th derivative of
a usual $C_0^\infty([0,1])$-function.
The value of $m$ will be determined later. We mention that the function
$g$ can be extended as a $C^\infty$-function to $\mathbb{R}$ by setting $g(t)=0$ for all $t \in \mathbb{R} \setminus [0,1]$. By multiplying with a constant, we may additionally
assume that $\|g\|_{H^k(0,1)} \leq 1$.
Note also that $\|g\|_{L^2(0,1)} = C_1 \not = 0$
and also $\|A g\|_{\ell^2} = C_2 \not = 0$ because of the
injectivity of $A$.

We now define for $0 < r \leq 1$ and $p \ge k -\frac{1}{2}$ the scaled function
\[ x_r(t):=  r^p g(\tfrac{t}{r}) \qquad (0 \le t \le 1). \]
Denoting by $ x_r^{(n)}$ the $n$-th derivative with respect to $t$,
we find  due to $0<r \le 1$ and $2 p + 1 \geq 2k$  that
\[ x_r^{(n)}(t)  = r^{p-n} g^{(n)} (\tfrac{t}{r}) \]
and
\begin{align*}
\| x_r^{(n)}\|_{L^2(0,1)}^2 &=
r^{2p-2n}  \int_0^1 g^{(n)} (\tfrac{t}{r})^2 dt =
r^{2p+1-2n}  \int_0^{r^{-1}} g^{(n)} (z)^2 dz\\
&=
r^{2p+1-2n}  \|g^{(n)} \|_{L^2(0,1)}^2
\le
r^{2(k-n)} \|g^{(n)}\|_{L^2(0,1)}^2
\end{align*}
because the support of $g$ is a subset of $[0,1]$.
Hence, we have
\[
\| x_r\|_{H^k(0,1)}^2 =
\sum_{n=0}^{k} \|x_r^{(n)}\|_{L^2(0,1)}^2
\le \sum_{n=0}^{k}r^{2(k-n)}   \|g^{(n)}\|_{L^2(0,1)}^2  \leq
\|g\|_{H^k(0,1)}^2 \leq 1.
 \]
Calculating the moments yields similar
\[ \int_0^1  x_r(t) t^{j-1}dt = r^p
 \int_0^1  g(\tfrac{t}{r}) t^{j-1} dt =
 r^{p+1} \int_0^1  g(z) (rz)^{j-1} dz =
 r^{p+j} \int_0^1  g(z) z^{j-1} dz.\]
 Thus
 \begin{align*}
 \|A x_r\|_{\ell^2}^2 &=
 \sum_{j=1}^\infty r^{2p + 2 j}
 \left(\int_0^1  g(z) z^{j-1} dz \right)^2
 =  \sum_{j=m+1}^\infty r^{2p + 2 j}
 \left(\int_0^1  g(z) z^{j-1} dz \right)^2\\
& \leq  r^{2p + 2  + 2m }
 \sum_{j=m+1}^\infty  \left(\int_0^1  g(z) z^{j-1} dz \right)^2
 \leq  r^{2p +2 + 2 m } \|A g\|_{\ell^2}^2.
 \end{align*}
Thus,
 \begin{align*}
\frac{ \|x_r\|_{L^2(0,1)}^2}{\|A x_r\|_{\ell^2}^{2\mu}} &\geq
\frac{   r^{2p+1} \|g\|_{L^2(0,1)}^2}{r^{(2p +2 + 2 m)\mu} \|A g\|_{\ell^2}^{2\mu}. }
= \frac{1}{r^{(1 + 2 m)\mu - (2p+1)(1-\mu)}}
\frac{\|g\|_{L^2(0,1)}^2}{\|A g\|_{\ell^2}^{2\mu}} \\
&=
\frac{1}{r^{(1 + 2 m)\mu - (2p+1)(1-\mu)}} \frac{C_1^2}{C_2^{2\mu}}.
\end{align*}
Now we may choose $m$ large enough as $m> (2p+1)\left(\frac{1-\mu}{\mu}\right)-\frac{1}{2}$ such that
the exponent $(1 + 2 m)\mu - (2p+1)(1-\mu)$ of $r$ in the denominator becomes positive. Then, we observe that for $r$ small enough
the right-hand side can be made arbitrary large  and in particular
larger  than any given constant $C$ in the proposition.
This proves the result.
\end{proof}

\begin{remark} \label{rem:svd}
{\rm For operator equations \eqref{eq:opeq} with compact linear operators $A$ in Hilbert spaces the degree of ill-posedness can be characterized by the decay rate of the singular values $\sigma_i$ of the forward operator. If we consider the equation $\mathcal{E}_k\, x=z$ for the embedding operator $\mathcal{E}_k: H^k(0,1) \to L^2(0,1)$, then it is well-known that $\sigma_i(\mathcal{E}_k) \sim i^{-k}$. This implies that finding the $k$-th derivative is a moderately ill-posed problems with ill-posedness degree $k$. Since the HMP operator $A$ from \eqref{eq:A} is not compact, we cannot verify its degree of
ill-posedness by means of singular values. However, the composition $A \circ \mathcal{E}_k: H^k(0,1) \to \ell^2$ is a compact operator, and one can observe the impact of the non-compact operator $A$ on the
compact embedding operator $\mathcal{E}_k$ by considering the decay rate  of $\sigma_i(A \circ \mathcal{E}_k)$. If one had a constant $C>0$ such that \begin{equation}\label{eq:stab}
\|x\|_{L^2(0,1)} \le C\, \|Ax\|_{\ell^2} \end{equation}
for all $x \in H^k(0,1)$, then we would have $\sigma_i(A \circ \mathcal{E}_k) \sim \sigma_i(\mathcal{E}_k)$ and the HMP operator $A$ would not have an impact on the degree of ill-posedness of the $k$-times differentiation problem. However,  Proposition~\ref{pro:mu} indicates that this is not
true. Precisely, the logarithmic rate occurring in Theorem~\ref{thm:log} indicates for $k=1$ that
the operator equation \eqref{eq:opeq} with $A \circ\mathcal{E}_1$ as forward operator is even severely
(exponentially) ill-posed. This is even more remarkable as we have
verified that the stability inequality~\eqref{eq:stab} {\em does} hold on
a subspace $X_1$! Contrary to expectation, however,
the composite operator  $A \circ \mathcal{E}_k$
does not seem to inlcude an infinite-dimensional subspace where it is mildly ill-posed.

Consequently, we conclude that {\em a non-compact operator with non-closed range can in a composition with a compact operator strongly destroy the ill-posedness degree of the compact part}.

Such question was discussed in \cite{HW09}, where in $L^2(0,1)$ the composition of a non-compact multiplication operator of type \eqref{eq:mult} with the compact integration operator had been studied. In contrast to the HMP situation, it could be shown in \cite{HW09} that wide classes of such multiplication operators with essential zeros in the multiplier function $m$ do not change the decay rate of the singular values of the integration operators and hence do not no change the degree of ill-posedness caused by the non-compact part.
}\end{remark}

\subsection{H\"older stabilty at $t = 1$} \label{sub:t1}
For the HMP, the reconstruction of one single value at $t=1$ is much
more stable than the reconstrution of the whole function.
We therefore look for conditional stability
estimates with $\|x\|_{L^2}$ replaced by the
evaluation functional $\ell: x\to x(1)$.

\medskip

We have the following result:
\begin{theorem}
Assume that $x \in H^1(0,1)$ with the a priori bound
\[ \|x^\prime\|_{L^2(0,1)} \leq E_1 .\]
Then, with $\delta$ defined in \eqref{eq:defdel},
we find
\[ |x(1)| \leq C  (E_1 \delta)^\frac{1}{2}  \]
and consequently the corresponding asymptotics
$$ \sup_{x \in H^1(0,1):\, \|x^\prime\|_{L^2(0,1)} \le E_1,\;\|Ax\|_{\ell^2} \le \delta\,} |x(1)| = \mathcal{O} \left(\sqrt{\delta}\right) \quad \mbox{as} \quad \delta \to 0.$$
Furthermore, with $x \in H^1(0,1)$ and the a priori bound
\[ \|x'\|_{L^\infty(0,1)} \leq E_\infty, \]
we get the stability estimate
$$|x(1)| \leq \tilde C \,\delta^{2/3}\,\ln(1/\delta)$$
with a constant $\tilde C>0$ depending on $E_\infty$ whenever $\delta>0$ is sufficiently small.
This yields the asymptotics
$$\sup_{x \in H^1(0,1):\, \|x^\prime\|_{L^\infty(0,1)} \le E_\infty,\;\|Ax\|_{\ell^2} \le \delta\,} |x(1)| = \mathcal{O} \left[ \delta^{2/3}\ln(1/\delta)\right] \quad \mbox{as} \quad \delta \to 0.$$
If, however, $\delta \ge \underline{c}>0$, then we have an estimate of the form
$$|x(1)| \leq \bar{C}\,\delta$$
with some constant $\bar{C}>0$ depending on $E_\infty$ and $\underline c$.
\end{theorem}
\begin{proof}
Since we suppose $x \in H^1(0,1)$, the point evaluation $x(1)$ is always well-defined by the trace theorem. Using integration by parts, we obtain for all
$j\geq 1$
\[ \int_{0}^1 x(t) t^{j-1} dt  =
\int_{0}^1 x(t) \frac{d}{dt}\left( \frac{1}{j} t^{j} \right) dt =
\frac{x(1)}{j}
-\int_{0}^1 x'(t)\frac{1}{j} t^{j} dt. \]
Thus,
\[ x(1) = \int_{0}^1 x'(t) t^{j} dt + j
\int_{0}^1 x(t) t^{j-1} dt  \]
Now we consider the first case characterized by an $L^2$-norm bound of the first derivative.
By the Cauchy-Schwarz inequality, we get that
\begin{align}\label{myugl}
\begin{split}
|x(1)| &\leq \|x'\|_{L^2(0,1)} \|t^j\|_{L^2(0,1)} + \left| j
\int_{0}^1 x(t) t^{j-1} dt \right| \\
& = \|x'\|_{L^2(0,1)} \frac{1}{\sqrt{2j+1}} + \left|j
\int_{0}^1 x(t) t^{j-1} dt \right|.
\end{split}
\end{align}
This equation hold for all $j$.
Summing up inequality \eqref{myugl} for $j = 1,N$, where $C$ denoting here a generic constant, yields
\begin{align*}
|x(1)| &=
\frac{1}{N} \sum_{j=1}^N |x(1)|
\leq
E_1 \frac{1}{N} \sum_{j=1}^N \frac{1}{\sqrt{2j+1}} +
\frac{1}{N} \sum_{n=1}^N j\left| \int_{0}^1 x(t) t^{j-1} dt \right| \\
& \leq
E_1 \frac{1}{N}  \sum_{j=1}^N
\int_{j-1}^{j}\frac{1}{\sqrt{2z+1}} dz +
 \frac{1}{N}  \left(\sum_{j=1}^N j^2 \right)^\frac{1}{2}
 \left( \sum_{j=1}^N (\int_{0}^1 x(t) t^{j-1} dt )^2
 \right)^\frac{1}{2} \\
 &\leq E_1 \frac{\sqrt{2 N+1} -1} {N} +
  \frac{C \sqrt{N^3}}{N}  \delta
 \leq C E_1 \frac{1}{\sqrt{N}} + C \sqrt{N}\delta.
\end{align*}
We may minimize the above expression by balancing
both terms which leads to the choice
 $N = \frac{E_1}{\delta}$. Thus, we obtain the first
 result
  \begin{align*}
 |x(1)|  \leq C  (E_1 \delta)^\frac{1}{2}.
\end{align*}

\smallskip

In the second case using $E_\infty$ we replace
\eqref{myugl} by
\begin{align}\label{myugl1}
\begin{split}
|x(1)| &\leq \|x'\|_{L^\infty(0,1)} \int_0^1 t^j dt  + \left| j \int_{0}^1 x(t) t^{j-1} dt \right|  = \|x'\|_{L^\infty(0,1)} \frac{1}{j+1} +
j \left|\int_{0}^1 x(t) t^{j-1} dt\right|.
\end{split}
\end{align}
Proceeding in the same way,
where $\frac{1}{\sqrt{2j+1}}$ is now replaced by
$\int_0^1 t^j dt = \frac{1}{j+1}$,   we find here by summation for $j=1,N-1$,
\begin{align*}
|x(1)|
&\leq
E_\infty \frac{1}{N-1} \sum_{j=1}^{N-1} \frac{1}{j+1}+
\frac{1}{N-1} \sum_{j=1}^{N-1} j\left| \int_{0}^1 x(t) t^{j-1} dt \right| \\
&\leq E_\infty  \frac{1}{N-1} \sum_{j=1}^{N-1} \frac{1}{j+1}+
  \frac{C \sqrt{(N-1)^3}}{N-1}  \delta  \\
& \leq  E_\infty  \frac{1}{N-1} \int_1^{N} \frac{1}{t} dt +
  C \sqrt{N-1}  \delta  \\
  &\leq
   E_\infty  \frac{1}{N-1}  \ln(N)
   +  C \sqrt{N}  \delta \le 2 E_\infty  \frac{1}{N}  \ln(N)
   +  C \sqrt{N}  \delta.
\end{align*}
Now we choose $N=\delta^{-2/3}$ and obtain for sufficiently small $\delta>0$ and some constant
$\tilde C>0$ depending on $E_\infty$ and $C$ that
$$|x(1)| \leq  2 E_\infty\, \delta^{2/3} \ln(\delta^{-2/3})+  C\, \delta^{2/3} =\frac{4}{3} E_\infty\,\delta^{2/3} \ln(\delta^{-1})+  C\, \delta^{2/3}
\le \tilde C \,\delta^{2/3}\,\ln(1/\delta).$$
Namely, we have $1 \le \ln(\delta^{-1})$ for $\delta \le 1/e$.

\smallskip

If, however, $\delta \ge \underline{c}>0$, then we have with $\|x^\prime\|_{L^\infty(0,1)} \le E_\infty$
and for $N=2$ that
$$|x(1)| \le E_\infty\,\ln(2)+C\sqrt{2}\,\delta \le \left(\frac{E_\infty \,\ln(2)}{\underline{c}}+C\sqrt{2}\right)\,\delta=\bar{C}\,\delta.$$
This completes the proof.
 \end{proof}

\section{Numerical case studies} \label{sec:casestudies}
\subsection{Numerical discussion of $\Lo^{-1}$}
One possible characterization of the range $\mathcal{R}(A)$ is given by \eqref{rangecond}. It is therefore of interest to study the operator $\Lo^{-1}$, and we will do so in the following in the discrete setting, where we consider truncations $\Lo^{-1}_n:=(\Lo^{-1}_{ij})_{i,j=1}^n$, $n\in\mathbb{N}$. The matrices $\Lo^{-1}_n$ are the Cholesky factors of the inverse Hilbert matrix $\mathcal{H}_n^{-1}$ introduced in \eqref{eq:vector}. Thus we have $\mathcal{H}_n^{-1}=\Lo^{-1}_n(\Lo^{-1}_n)^T$. In particular, this means that $\|\Lo^{-1}_n\|_2^2=\|\mathcal{H}_n^{-1}\|_2$ where $\|\cdot\|_2$ is the largest singular value of the matrix. Consequently, we have the bound $\|\Lo^{-1}_n\|\leq  C \exp(1.763n)$ from \eqref{eq:hilbertmatrx_norm}. This asymptotics can be confirmed numerically as shown in Figure \ref{fig:Linvgrowth}. It is interesting that the asymptotics $C\exp(1.763n)$ not only describes the norms $\|\Lo^{-1}_n\|$, but also the row-wise absolute maxima of the matrices. Because the entries of $|\Lo|^{-1}_n$ have alternating sign, we set $|\Lo|^{-1}_n=|(\Lo^{-1}_n)_{ij}|$, $i,j=1,\dots,n$. Figure \ref{fig:Linvgrowth} demonstrates that
$$
\max_{j=1,\dots, i} (|\Lo|^{-1}_n)_{ij}\leq C\exp(1.763n), \quad i=1,\dots n.
$$
The maxima are not found on the main diagonal $(|\Lo|^{-1}_n)_{ii}$, $i=1,\dots,n$, but for some $[\frac{i+1}{2}]<j<i$. On the main diagonal itself, the entries grow with a slightly lower, approximately $(|\Lo|^{-1}_n)_{ii}\leq C\exp(1.4 i)$. A plot of selected rows of $|\Lo|^{-1}_n$ is given in Figure~\ref{fig:Lrows}. Clearly, the entries grow fast both in row index $i$ and in column index $j$. Hence, one would expect that sequences $\{y_i\}_{i=1}^\infty$ satisfying the range condition \eqref{rangecond} have to decay rapidly, which is in contrast to the slowly decaying example at the end of Section 3. This seeming contradiction is resolved because the entries of $\Lo_n^{-1}$ have alternating sign, which means that the sums
\[
\sum_{j=1}^i (\Lo_n^{-1})_{ij} y_j, \quad i=1,2,\dots
\]
do not necessarily explode, for example when the elements of $\{y_i\}_{i=1}^\infty$ have constant sign as is the case for the example \eqref{eq:y}. This suggests that monotonicity might play a crucial role in the characterization of the stable subspace $Y_1$.

\begin{figure}
\includegraphics[width=\linewidth]{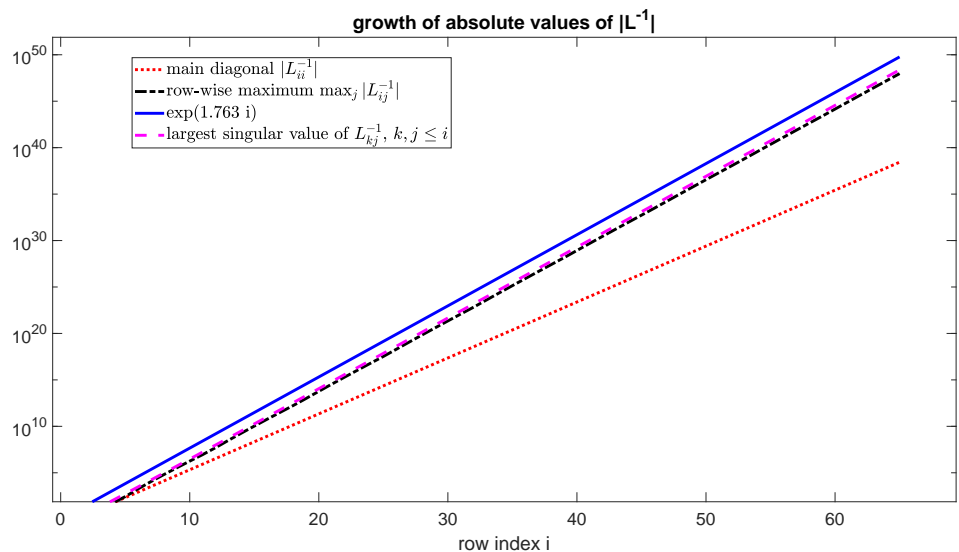}\caption{Numerical studies for $\Lo^{-1}_n$, $n=65$. Blue, solid: asymptotic bound $\exp(1.763 i)$. Magenta, dashed: norms $\|\Lo^{-1}_i\|_2$. Black, dash-dotted: row-wise maxima of $|\Lo|^{-1}_n$. Red, dotted: main diagonal of $|\Lo|^{-1}_n$.}\label{fig:Linvgrowth}
\end{figure}

\begin{figure}
\includegraphics[width=\linewidth]{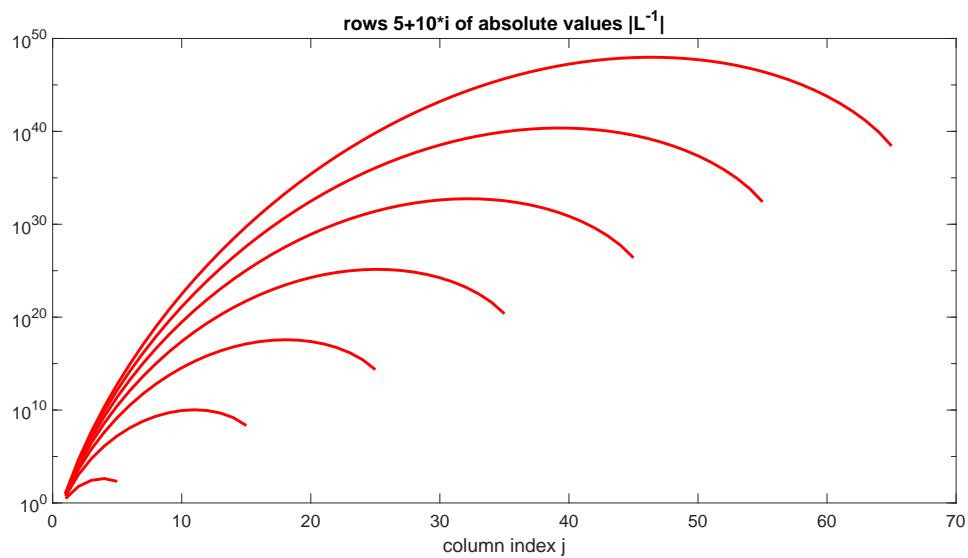}\caption{Numerical studies for $|\Lo|^{-1}_n$, $n=65$. Plot of rows $5+10k$, $k=1,\dots,6$.}\label{fig:Lrows}
\end{figure}
\newpage
\subsection{Numerical verification of the noise amplification factor}
The following case study illustrates and complements the results of Remark~\ref{rem:rem3} and in particular of the upper estimate~\eqref{eq:Andagger_bound}. We show that the regularization error in~\eqref{eq:est1} is driven by the amplification factor $\|A_n^\dagger\|_{\mathcal{L}(\ell^2,L^2(0,1))}$ and that the upper limit~\eqref{eq:Andagger_bound} and the resulting rate in n presents a reasonable bound in practical situations.
Therefore we introduce a test case with exact solution
\begin{equation*}
\tilde x(t)=0.2+ \frac{0.36}{1+100(2.05t-0.2)^2}\quad (0 \le t \le 1).
\end{equation*}
Because $\tilde x$ is almost constant for $t>0.5$ with $\tilde x(t)\approx 0.2$, it is easy to see that $[A\tilde x]_i=\mathcal{O}(\frac{0.2}{i})$.
In the next step we calculate for a sample of noisy data and fixed $n$ the minimum-norm solutions in the noisy case
\begin{equation*}
x_n^\delta:=A_n^\dagger y^\delta={\rm argmin}\{\|x\|_{L^2(0,1)}|\, x \in L^2(0,1): \|A_n x -  y^\delta\|_{\ell^2}=\min\}
\end{equation*}
and the associated version $x^\dagger_n:=A_n^\dagger y$ for the noise-free case ($\delta=0)$.
As the exact solution is assumed to be known, we can can compute the regularization errors \linebreak $\|x^\delta_n-x_n^\dagger\|_{L^2(0,1)}$.
Subsequently we perform a linear regression in accordance with Remark~\ref{rem:rem3} for decaying noise level and fixed n, i.e., we assume
\begin{equation*}
\|x^\delta_n-x_n^\dagger\|_{L^2(0,1)} \approx \|A_n^\dagger\|_{\mathcal{L}(\ell^2,L^2(0,1))} \delta.
\end{equation*}
As a consequence we receive estimators for the amplification factor $\|A_n^\dagger\|_{\mathcal{L}(\ell^2,L^2(0,1))}$ and denote these estimators with $f_n$. In order to improve the accuracy of the results multiple realizations of the error were used. A regression was performed for each realization. The presented results are the mean of these regressions. This behaviour is visualized in Figure~\ref{fig:chris1}.
In this context, Figure~\ref{fig:chris2} visualizes the quotient
$\ln(f_n)/n$ for various $n$, which can be interpreted as the factor to $n$ in the exponent of the estimation~\eqref{eq:Andagger_bound}.  We conclude that the numerical observations coincide with the previously introduced theoretical findings and the resulting rates match.
\begin{figure}
\centering
\includegraphics[width=\linewidth]{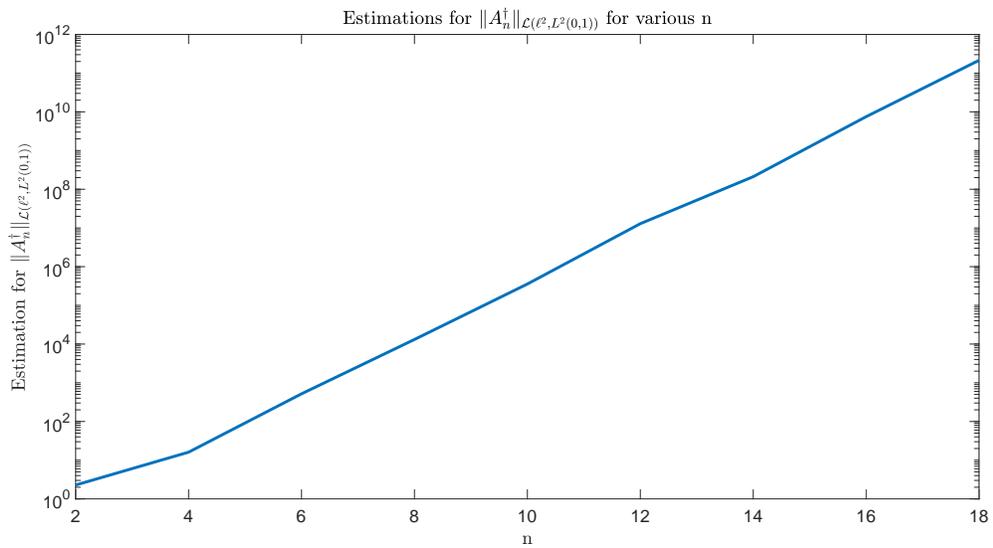}\caption{Estimations for the amplification factor $\|A_n^\dagger\|_{\mathcal{L}(\ell^2,L^2(0,1))}$ for various n.}
\label{fig:chris1}\end{figure}
\begin{figure}
\centering
\includegraphics[width=\linewidth]{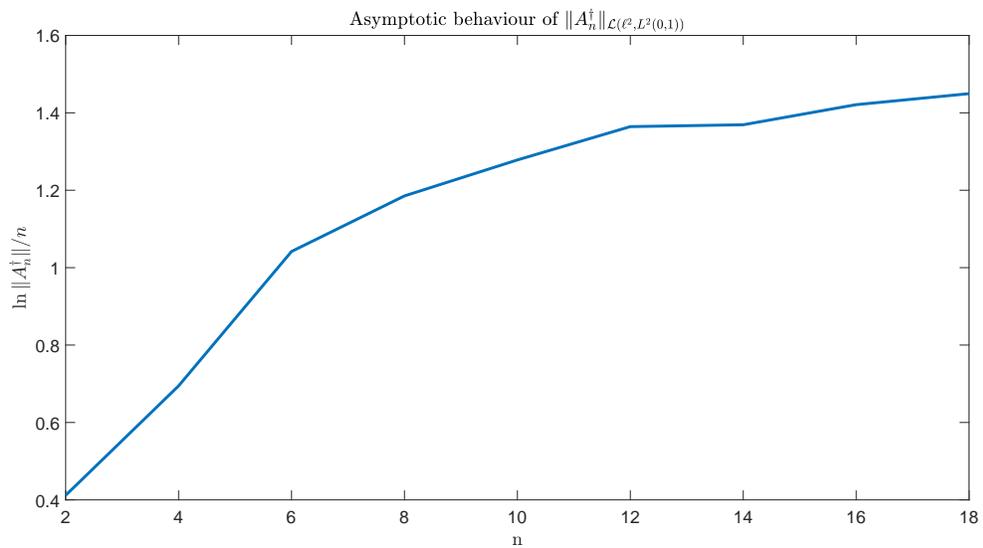}\caption{Quotient $\ln(f_n)/n$ for estimations of the amplification factor $\|A_n^\dagger\|_{\mathcal{L}(\ell^2,L^2(0,1))}$ and various n.}
 \label{fig:chris2}\end{figure}
\section*{Acknowledgement}
BH and CH have been supported by the German Science Foundation (DFG) under the grant
HO 1454/12-1, Project Number 391100538. DG has been supported by the German Science Foundation (DFG) under the grant GE 3171/1-1, Project Number 416552794.


\begin{flushleft}

Daniel Gerth,\\
Chemnitz University of Technology, \\
Faculty of Mathematics, 09107 Chemnitz, Germany,\\
Email: {\tt daniel.gerth@mathematik.tu-chemnitz.de},\\

\smallskip

Bernd Hofmann,\\
Chemnitz University of Technology, \\
Faculty of Mathematics, 09107 Chemnitz, Germany,\\
Email: {\tt bernd.hofmann@mathematik.tu-chemnitz.de},\\

\smallskip

Christopher Hofmann,\\
Chemnitz University of Technology, \\
Faculty of Mathematics, 09107 Chemnitz, Germany,\\
Email: {\tt christopher.hofmann@mathematik.tu-chemnitz.de},\\

\smallskip

Stefan Kindermann,\\
Johannes Kepler University Linz, Industrial Mathematics Institute,\\
Altenbergerstra{\ss}e 69, A-4040 Linz, Austria,\\
Email: {\tt kindermann@indmath.uni-linz.ac.at}

\end{flushleft}

\end{document}